\documentclass[10pt]{article}
\usepackage[margin=1.15in]{geometry}
\usepackage{amsfonts, amsmath, amssymb,latexsym,amsthm,color,setspace}

\newtheorem{defn0}{Definition}[section]
\newtheorem{prop0}[defn0]{Proposition}
\newtheorem{thm0}[defn0]{Theorem}
\newtheorem{lemma0}[defn0]{Lemma}
\newtheorem{corollary0}[defn0]{Corollary}
\newtheorem{example0}[defn0]{Example}
\newtheorem{remark0}[defn0]{Remark}
\newtheorem{conjecture0}[defn0]{Conjecture}
\newtheorem{discusion0}[defn0]{Discusion}

\newenvironment{definition}{ \begin{defn0}}{\end{defn0}}
\newenvironment{proposition}{\bigskip \begin{prop0}}{\end{prop0}}
\newenvironment{theorem}{\bigskip \begin{thm0}}{\end{thm0}}
\newenvironment{lemma}{\bigskip \begin{lemma0}}{\end{lemma0}}
\newenvironment{corollary}{\bigskip \begin{corollary0}}{\end{corollary0}}
\newenvironment{example}{ \begin{example0}\rm}{\end{example0}}
\newenvironment{remark}{ \begin{remark0}\rm}{\end{remark0}}

\newcommand{\propref}[1]{Proposition~\ref{#1}}
\newcommand{\thmref}[1]{Theorem~\ref{#1}}
\newcommand{\lemref}[1]{Lemma~\ref{#1}}
\newcommand{\corref}[1]{Corollary~\ref{#1}}
\newcommand{\exref}[1]{Example~\ref{#1}}

\def\series#1{{{\mathbb C}[\![x_1,\dots x_{#1}]\!]}}
\def\max{{\bf m}}                   
\def\res{{\bf k}}                   
\def\length{{\mathrm{ Length}}}
\def\Hom{{\mathrm{ Hom}}}
\def\HF{{\operatorname{H\!F}}}
\def\HP{\operatorname{H\!P}}
\def\pn{\operatorname{pn}}
\def\soc{\operatorname{soc}}

\def\spec{{\bf Spec}}
\def\hilb{{\bf Hilb}}

\def\ext{{\mathrm{ Ext}}}

\def\Can{{\mathrm{ Can }}}
\def\Pf{{\mathrm{ Pf }}}

\def\OX{{\mathcal O}_{X}}

\def\OX0{{\mathcal O}_{(X,0)}}
\def\OCn0{{\mathcal O}_{(\mathbb C^n,0)}}
\def\Cn0{({\mathbb C}^n,0)}
\def\OC30{{\mathcal O}_{(\mathbb C^3,0)}}
\def\C30{({\mathbb C}^3,0)}
\def\val{ {\rm val} }

\def\cocoa
{\mbox{\rm C\kern-.13em o\kern-.07 em C\kern-.13em o\kern-.15em A}}

\title{  \bf \huge On the canonical ideals of one-dimensional Cohen-Macaulay local rings
}
\author{\large   Juan Elias
\thanks{Partially supported by  MTM2010-20279-C02-01 and  PRX12/00173}
 }

\date{\today}

\begin{document}

\maketitle

\begin{abstract}
In this paper we consider the problem of finding explicitly canonical ideals of one-dimensional Cohen-Macaulay local rings.
We show that Gorenstein ideals contained in a high power of the maximal ideal are canonical ideals.
In the codimension two case, from a Hilbert-Burch resolution, we show how to construct canonical ideals of curve singularities.
Finally, we   translate the problem of the analytic classification of
curve singularities  to the classification of local Artin Gorenstein rings  with suitable length.
\end{abstract}

\section{Introduction}

Let $(R,\max)$ be a one-dimensional  Cohen-Macaulay local ring with maximal ideal $\max$
for which there exists a canonical module $\omega_R$, this is the case, for instance,  if $R$ is the quotient of a local Gorenstein ring.
Recall that $R$ possesses a   canonical ideal (more precisely, the canonical module $\omega_R$ of $R$ exists and is contained in $R$) if and only if the total ring of fractions of the $\max$-adic completion of $R$ is Gorenstein, \cite[Satz 6.21]{HK71}.
See \cite[Chapter 3]{BH97} for the basic properties of canonical modules and canonical ideals.

The question that motivates this paper is: can we describe explicitly canonical ideals?
Recall that
Boij in \cite{Boi99b} addressed this problem  for projective zero-dimensional schemes.
A second question that we consider is: can we use canonical ideals for the analytic classification of singularities?
In this paper we study these questions in the one dimensional case.

The contents of this paper is the following.
In the second  section we show that Gorenstein ideals contained in a high power of the maximal ideal are canonical ideals, \propref{existence}.
In the codimension two case, and following \cite{Boi99b},  from a Hilbert-Burch resolution we show how to construct canonical ideals for curve singularities, \propref{anti}.
In this section we recall how to describe "explicitly"   the canonical module  by using Rosenlicht's  regular differential forms, \cite[IV 9]{Ser59}.
This strategy is very useful in the case  of branches, and  specially in the case case of monomial curve singularities.
In the third section we  address the problem of the analytic classification of curve singularities.
We show  canonical ideals $I\subset \OX0$, where $X$ is a curve singularity,
for which we compute the multiplicity and the socle degree of the Artin Gorenstein quotient $\OX0 /I$,
\propref{canlocus}.
In \thmref{chargor} we  translate the problem of the analytic classification of
 curve singularities $X$ to the classification of local Artin Gorenstein rings $\OX0 /I$ of suitable length.


\medskip
\noindent {\sc Notations:}
Let $(R,\max)$ be a one-dimensional  Noetherian local ring with maximal ideal $\max$ and residue filed $\res=R/\max$.
If $I$ is a $\max-$primary ideal of $R$ we denote by
$\HF_{I}^1(n)=length_R(R/I^{n+1})$ the Hilbert-Samuel function of $I$.
Hence there exist integers $e_0=(I)\ge 1$ and $e_1(I)$ such that
$\HP_{I}^1(n):=e_0(I) (n+1)-  e_1(I)$, this is the  Hilbert-Samuel polynomial of $I$, i.e. $\HF_{I}^1(n)=\HP_{I}^1(n)$
for $n \gg 0$.
The integer $e_0(I)$  is the multiplicity of $I$.
We denote by $\HF_{I}^0(n)=length_R(I^n/I^{n+1})$ the $0$-th Hilbert-Samuel function of $I$.
Then $\HP_{I}^0(n)=e_0(I)$ for $n \gg 0$, the postulation number of $I$  is the least integer $\pn(I)$ such that
 $\HF_{I}^0(n)=\HP_{I}^0(n)$ for all $n\ge \pn(I)$.
We set $\HF_R^i=\HF_{\max}^i$,  $\HP_R^i=\HP_{\max}^i$, $i=0,1$, and $\pn(R)=\pn(\max)$.
If $M$ is an Artin $R$-module we define the socle degree $s(M)$ of $M$ as the last integer $t$ such that
$\max^t M \neq0$.
The socle of $M$ is by definition the $R$-submodule  $\soc(M)=(0:_M\max)$ of $M$.

Next we recall some basic facts of curve singularities.
Let $(X,0)$ be a reduced curve singularity of $(\mathbb C^n, 0)= \spec(\OCn0)$, i.e.
$\OCn0=\series{n}$, and
$(X,0)=\spec(\OX0)$ where
$\OX0= \OCn0/I_X$ is a one-dimensional reduced ring with maximal ideal $\max_X$.
We write $\HF^i_X=\HF^i_{\OX0}$ and  $\HP^i_X=\HP^i_{\OX0}$.
We assume that $n$ is the embedding dimension of $(X,0)$, this is equivalent to say
that $\max_X  / \max_X^2$ is isomorphic as $\mathbb C$-vector space to the homogeneous
linear forms of
$P=\mathbb C[X_1,\cdots,X_n]$.
Hence all element $x\in \max_X  / \max_X^2$ defines an element  in the quotient $\OX0$ that
we will denote again by $x$.
Let $\nu: \overline{X}=\spec (\overline{\OX0})\longrightarrow (X,0)$ be the normalization of $(X,0)$, where
$\overline{\OX0}$ is the integral closure of $\OX0$ on its full ring ${\mathrm{ tot}}(\OX0)$ of fractions.
The singularity order of $(X,0)$ is
$
\delta(X)=\dim_{\mathbb C}\left({\mathcal O}_{\overline{X}}/\OX0\right).
$
We denote by $\mathcal C$ the conductor of the finite extension
$\nu^*: \OX0 \hookrightarrow
\overline{\OX0}$
and by $c(X)$ the dimension of $\overline{\OX0}/\mathcal C$.
Let  $\omega_{(X,0)}=\ext ^{n-1}_{\OCn0}(\OX0, \Omega^n_{\OCn0})$ be the dualizing module of $(X,0)$.
We can consider the composition morphism of $\OX0$-modules
$$
\gamma_X: \Omega_{(X,0)} \longrightarrow
\nu_* \Omega_{\overline{X}}\cong
\nu_* \omega_{\overline{X}} \longrightarrow
 \omega_{(X,0)}.
$$
Let $d: \OX0 \longrightarrow  \Omega_{(X,0)}$ the universal derivation, then we have a $\mathbb C$-map $\gamma_X  d$ that we also denote
by $d: \OX0 \longrightarrow  \omega_{(X,0)}.$
The Milnor number of $(X,0)$ is  $\mu(X)=\dim_{\mathbb C}(\omega_{(X,0)}/d \OX0)$, \cite{BG80}.
Notice that $(X,0)$ is non-singular iff $\mu(X)=0$ iff $\delta(X)=0$ iff $c(X)=0$.

In the following result we collect some basic results on $\mu$ and other numerical invariants that
we will use later on.

\begin{proposition}
\label{basic}
Let $(X,0)$ be a reduced curve singularity of embedding dimension $n$.
Then

\noindent
$(i)$ $\mu(X)= 2 \delta(X) -r+1$, where $r$ is the number of branches of $(X,0)$.

\noindent
$(ii)$
It holds
$$
 e_0(X)-1\le e_1(X) \le \delta(X)\le \mu(X)
$$
and $e_1(X)\le {e_0(X)\choose 2} -{ n-1\choose 2}$.

\noindent
$(iii)$ If $X$ is singular then $\delta(X)+1 \le c(X)\le 2 \delta(X)$, and $c(X) =  2 \delta(X)$
if and only if $\OX0$ is a Gorenstein ring.
\end{proposition}
\begin{proof}
$(i)$  \cite[Proposition 1.2.1]{BG80}.
$(ii)$ \cite[Proposition 1.2.4 (i)]{BG80}, \cite{Nor59a},  \cite{Eli90}, \cite{Eli01}.
$(iii)$ \cite[Proposition 7, pag. 80]{Ser59}, and \cite{BC77}.
\end{proof}

\medskip
\noindent
{\sc ACKNOWLEDGMENTS.} The author thanks the referee for improving some results of section two.

\section{Canonical ideals}

The first aim of this section is to find conditions on a $\max$-primary ideal $I$ to be a canonical ideal.

\begin{lemma}
\label{push}
Let $(R,\max)$ be a one-dimensional Cohen-Macaulay local ring and let $I$ be an $\max$-primary ideal of $R$.
Let $x\in\max$ be a parameter of $R$.
\begin{enumerate}
\item[(i)] If $R/x^n I$ is a Gorenstein ring for some $n\ge 1$, then $R/I$ is a Gorenstein ring and
$(x^n I:_R\max)=x^n(I:_R\max)$.
\item[(ii)] Assume that $I\subset xR$. If $R/I$ is a Gorenstein ring, then $R/x^nI$ is a Gorenstein
ring for all $n\ge 1$.
\end{enumerate}
\end{lemma}
\begin{proof}
\noindent
$(1)$ The short exact sequence
$$
0\longrightarrow
R/I
\stackrel{x^n}{\longrightarrow}
R/x^nI
\longrightarrow
R/x^nR
\longrightarrow 0
$$
yields the exact sequence
$$
0\longrightarrow
\Hom_R(R/\max,R/I)
\stackrel{x^n}{\longrightarrow}
\Hom_R(R/\max,R/x^nI)
\longrightarrow
\Hom_R(R/\max,R/x^nR)
$$
of socles, which shows
$$
\Hom_R(R/\max,R/I)
\stackrel{x^n}{\cong}
\Hom_R(R/\max,R/x^nI),
$$
because $\Hom_R(R/\max,R/I)\neq 0$ and $\length_R(\Hom_R(R/\max,R/x^nI))=1$.
Hence $R/I$ is a Gorenstein ring.

\noindent
$(ii)$
Let $\alpha\in (x^n I:_R\max)$.
Then since $x \alpha\in x^n I\subset x^{n+1}R$, we get $\alpha\in x^n R$,
which shows $\pi((x^n I:_R\max)/x^nI)=0$
where $\pi:R/x^nI\longrightarrow R/x^nR$ denotes the canonical epimorphism.
Hence we get the isomorphism
$$
\Hom_R(R/\max,R/I)
\stackrel{x^n}{\cong}
\Hom_R(R/\max,R/x^nI),
$$
in the exact sequence
$$
0\longrightarrow
\Hom_R(R/\max,R/I)
\stackrel{x^n}{\longrightarrow}
\Hom_R(R/\max,R/x^nI)
\longrightarrow
\Hom_R(R/\max,R/x^nR)
$$
of socles.
Thus
$\length_R(\Hom_R(R/\max,R/x^nI))=1$, so that $R/x^nI$ is a Gorenstein ring for all $n\ge 1$.
\end{proof}

\begin{proposition}
\label{push2}
Let $I$ be an $\max$-primary ideal of $R$.
Then $I$ is a canonical ideal of $R$ if and only if there exists a parameter $x\in\max$ of $R$ such that $R/xI$
is a Gorenstein ring.
\end{proposition}
\begin{proof}
We have only to prove the ${\it if\; part}$.
Let $Q$ the total ring of fractions of $R$.
Since ${\rm H}^1_{\max}(I)\cong Q/I$, it suffices to see
$\length_R(I:_Q\max)/I=1$.
Let $\alpha\in (I:_Q\max)$.
Then since $x\alpha\in I\subset R$ and $(x\alpha)\max \subset xI$, we have
$$
x\alpha\in (xI:_R\max)=x(I:_R\max)
$$
(see the proof of \lemref{push} (i)).
Hence $I\subsetneq (I:_Q\max)\subset (I:_R\max)$, so that
$$
\length_R(I:_Q\max)/I=\length_R(I:_R\max)/I=1
$$
as wanted.
\end{proof}

In the next result we prove that an $\max$-primary Gorenstein ideal contained in a high  power of the maximal ideal is a canonical ideal.
Notice that this result cannot be extended to any $\max$-primary Gorenstein ideal.
Let $R$ be a one-dimensional Cohen-Macaulay local ring of Cohen-Macaulay type $2$ and embedding dimension $b\ge 3$,
for instance $R=\res[\![t^5, t^6, t^7]\!]$.
 Then the maximal ideal is a Gorenstein ideal minimally generated by $b\ge 3$ elements.
Since the minimal number of generators of a canonical ideal is the Cohen-Macaulay type of $R$, $\max$ is not a canonical ideal.

\begin{corollary}
\label{existence}
Let $x\in \max$ and assume that $\max^{r+1}=x\max^r$ for some $r\ge 0$.
Let $I$ be an $\max$-primary ideal of $R$ such that $I\subset \max^{r+1}$.
If $R/I$ is a Gorenstein ring, then $I$ is a canonical ideal of $R$, whence ${\rm tot}(\widehat R)$ is a Gorenstein ring.
\end{corollary}
\begin{proof}
The assertion follows from \lemref{push} and \propref{push2}, since
$I\subset \max^{r+1}\subset x R$.
\end{proof}

\begin{remark}
\label{upperpost}
Recall  that
 $\pn(R)\le e_0(R)-1$,  \cite[Proposition 12.14]{Mat77}.
 Hence,
 if $I\subset \max^{e_0(R)}$ is an $\max$-primary ideal  such that
$R/I$ is a Gorenstein ring then $I$ is a canonical ideal.
 \end{remark}

Last result pointed out that a basic problem in Commutative Algebra  is to find methods to construct Gorenstein ideals.
We know that complete intersection ideals are Gorenstein; by a result of Serre,
in codimension two to be Gorenstein is equivalent to be a complete intersection; in codimension three, Gorenstein ideals are the ideals generated by the Pfaffians of  skew-symmetric matrices, \cite{BE77}.
In \propref{anti} we show how to construct canonical ideals, that are Gorenstein, from a
Hilbert-Burch resolution.
On the other hand, notice  that if $I$ is a canonical ideal and  $y\in \max$ is a non-zero divisor of $R$
then $y^tI$, $t\ge 1$, is a canonical ideal as well,  but
the length
of $R/ yI$ is not under control.
In fact, for all $t\ge 1$ we have, \cite[Theorem 12.5]{Mat77},
\begin{eqnarray*}
  \length_{R}(R/ y^tI)&=& \length_{R}(R/ I)+ \length_{R}(I/ y^t I) \\
  &=& \length_{R}(R/ I)+ t \length_{R}(R/ (y))\\
  &\ge& \length_{R}(R/ I)+ t e_0(R).
\end{eqnarray*}

In the next result  we will  find
canonical ideals for which we compute  the multiplicity or the socle degree;
in the second part  we take  $t\ge 4\mu(X)+1$, where $(X,0)$ is a reduced curve singularity,
because  we have to consider
a high $t$ in \thmref{chargor}.
See \exref{exp-example} for an explicit application of the next result.

\begin{proposition}
\label{canideal}
Let $(X,0)$ be a reduced curve singularity.

\noindent
$(i)$
Let $z$ be a degree $t\ge 2 \mu(X)+1$ superficial element of $\OX0$.
Then  the $\OX0$-module $z \omega_{(X,0)}$
is a canonical ideal of $(X,0)$ such that $\OX0/ z \omega_{(X,0)}$ is a Gorenstein ring of colength $t  e_0(X) -  2 \delta(X)$.

\noindent
$(ii)$
For  $t\ge 4 \mu(X)+1$ the socle degree of  $\OX0/ z \omega_{(X,0)}$ is
at most
$$
e_0(X)(t-2 \mu(X)-1) + 2 \delta(X) +e_1(X)+ 2(1-r).
$$
This number is bounded above by $\delta(X)(4 e_0(X) +3)$.

\noindent
$(iii)$
If $(X,0)$ is Gorenstein then $\OX0$ is a canonical ideal and for every superficial element $z$ of degree
$t\ge 1$, $ z \OX0$ is a canonical
ideal.
\end{proposition}
\begin{proof}
$(i)$
Since $\OX0$ is a one-dimensional reduced ring we know that $\omega_{(X,0)}$
is a sub-$\OX0$-module of $\mathrm{tot}(\OX0)$, \cite[3.3.18]{BH97}.
Let us consider the perfect pairing, \cite[Chapter IV]{Ser59},
$$
\begin{array}{ccccc}
  \frac{\nu_* {\mathcal O}_{\overline{X}}}{\OX0} & \times& \frac{\omega_{(X,0)}}{\nu_* \Omega_{\overline{X}}} & \stackrel{\eta}{\longrightarrow}  & \mathbb C \\
  F & \times  & \alpha & \longrightarrow  &  \sum_{i=1}^r {\mathrm{res}}_{p_i}(F \alpha)
\end{array}
$$
notice that for all $\lambda\in R$ it holds
$$
\eta(\lambda F, \alpha)= \sum_{i=1}^r {\mathrm{res}}_{p_i}(\lambda F \alpha)=\eta( F, \lambda\alpha).
$$
Hence, since $\max^{c(X)}$ annihilates the quotient  $\nu_* {\mathcal O}_{\overline{X}}/ \OX0$ we get that
$\max^{c(X)}$  annihilates also the quotient
$\omega_{(X,0)}/\nu_* \Omega_{\overline{X}}$.
Hence $\max^{c(X)} \omega_{(X,0)} \subset \nu_* \Omega_{\overline{X}} =  \nu_* {\mathcal O}_{\overline{X}}$.
Again, since $\max^{c(X)} \nu_* {\mathcal O}_{\overline{X}} \subset \OX0$ we get
$\max^{ 2c(X)} \omega_{(X,0)}  \subset \OX0$.
On the other hand the epimorphism of $\OX0$-modules
$$
\frac{\omega_{(X,0)}}{\OX0 d \OX0} \longrightarrow \frac{\omega_{(X,0)}}{\nu_* \Omega_{\overline{X}}}
$$
assures that
$\max^{\mu(X)} \omega_{(X,0)} \subset \nu_* \Omega_{\overline{X}} =  \nu_* {\mathcal O}_{\overline{X}}$.
Hence $\max^{c(X)}\nu_* {\mathcal O}_{\overline{X}} \subset \OX0$
$$
\max^{\mu(X)+c(X)} \omega_{(X,0)} \subset \OX0.
$$
Let us consider the sequence
$$
\max^t \omega_{(X,0)} \subset  \OX0 \subset \nu_* {\mathcal O}_{\overline{X}}=\nu_* \Omega_{\overline{X}} \subset \omega_{(X,0)}.
$$
From the perfect  pairing of the beginning of the proof we get $\dim_{\mathbb C}( \omega_{(X,0)}/ \OX0)=2 \delta(X)$.
Since $z$ is a degree $t\ge 2 \mu(X)$ superficial element of $\OX0$
$$
\dim_{\mathbb C}(\omega_{(X,0)} /z \omega_{(X,0)})=t e_0(X)
$$
\cite[Theorem 12.5]{Mat77}.
From this identity and \propref{basic}  we get the first part of the claim.

\noindent
$(ii)$
We have the following inequalities for $t\ge 4 \mu(X)+1$
\begin{eqnarray*}
s\left(\frac{\OX0}{z \omega_{(X,0)}}\right) &=&  s\left(\frac{\max^{2 \mu(X)+1}}{z \omega_{(X,0)}}\right) + 2 \mu(X) +1 \\
&\le&\length\left(\frac{\max^{2 \mu(X)+1}}{z \omega_{(X,0)}}\right) + 2 \mu(X) \\
&=& \length\left(\frac{\OX0}{z \omega_{(X,0)}}\right)-
\length\left(\frac{\OX0}{\max^{2 \mu(X)+1}}\right)+  2 \mu(X).
\end{eqnarray*}
From \propref{basic} and the first part of this result we get
\begin{eqnarray*}
s\left(\frac{\OX0}{z \omega_{(X,0)}}\right)  &\le& (t e_0(X)- 2 \delta(X))  - (e_0(2 \mu(X)+1)-e_1(X)) + 2 \mu(X)\\ \\
&=&e_0(X)(t-2 \mu(X)-1) + 2 \delta(X) +e_1(X)+ 2(1-r).
\end{eqnarray*}
From \propref{basic} we get that the socle degree
is bounded above by $\delta(X)(4 e_0(X)+3).$

\noindent
$(iii)$
Since any superficial element is a non-zero divisor we get the claim.
\end{proof}


Recall that it is possible to give an "explicit" description of  $\omega_{(X,0)}$ by using Rosenlicht's  regular differential forms,
\cite[IV 9]{Ser59},  \cite[Section 1]{BG80}.
This strategy is very useful in the case  of branches, specially in the  case of monomial curve singularities.

We denote by $\Omega_{\overline{X}}(p_{\cdot})$ the set of meromorphic forms in $\overline{X}$
with a pole at most in the set $\{p_1,\cdots,p_r\}$.
Then Rosenlicht's differential forms are defined as follows:
$\omega^R_{(X,0)}$ is the set of $\nu_*(\alpha)$,  $\alpha \in \Omega_{\overline{X}}(p_{\cdot})$, such that
for all $F\in \OX0$
$$
\sum_{i=1}^r {\rm res}_{p_i}(F \alpha)=0.
$$
Notice that we have a mapping that we also denote by
$
d_R : \OX0 \longrightarrow
\omega_{(X,0)} \longrightarrow
\nu_* \Omega_{\overline{X}} \hookrightarrow
 \omega^R_{(X,0)}.
$
In  \cite[Chap. VIII]{AK70} it  is proved that $\omega_{(X,0)} \stackrel{\phi}{\cong} \omega_{(X,0)}^R$ and
$d_R=\phi d$, where $d: \OX0 \longrightarrow  \omega_{(X,0)}$, see Section 1.
From now on we assume that $(X,0)$ is a branch, i.e. $r=1$.
Let $t\in {\rm tot}(\OX0)$ be a uniformizing   parameter of $(X,0)$, this means
 $\overline{\OX0}\cong \mathbb C [\![t ]\!]$.
We can  consider $\OX0$ as a sub-$\mathbb C$-algebra of
$\mathbb C [\![t ]\!]$ and then we may assume hat there exists a parametrization
of $(X,0)$
$$
\left\{
  \begin{array}{l}
    x_1=t^{n_1} \\
    x_i=f_i(t)\quad i=2,\cdots, n
  \end{array}
\right.
$$
with $n_1=e_0(X)$ and $\val_t(f_i)\ge n_1$, $i=2,\cdots,n$.
Here $\val_t$ denotes the valuation with respect $t$ that is defined in ${\rm tot}(\OX0)$.
Notice that the conductor $\mathcal C$ of the extension $\OX0\subset \overline{\OX0}$
is in particular an ideal of $\overline{\OX0}\cong \mathbb C [\![t ]\!]$ so it is generated by $t^c$.

For all subset $N$ of ${\rm tot}(\OX0)$ we denote
by $\Gamma_N$ the
the set of rational numbers  $\val_t(a)$ for all $a\in N\setminus \{0\}$, we assume that $\Gamma_N$ contains the zero element.

\medskip
The following result is well known. We include it here for the readers' convenience, see
\cite[Example 2.1.9]{GWI78}.

\begin{proposition}
Let $(X,0)$ be a branch.
It holds $\Gamma_{\omega_{(X,0)}}\subset \mathbb Z \setminus (-\Gamma_X - 1)$,
in particular $\Gamma_{\omega_{(X,0)}}\subset -c +\mathbb N$ and there exists
$\alpha \in \omega_{(X,0)}$ such that $\val_t(\alpha)=-c$.
If $(X,0)$ is a monomial curve singularity then $\Gamma_{\omega_{(X,0)}}= \mathbb Z \setminus (-\Gamma_X - 1)$.
\end{proposition}
\begin{proof}
All  $\alpha \in \omega_{(X,0)}$ can be written $\alpha = t^n g(t) d t$ with $n\in \mathbb Z$ and
$g(t)\in \mathbb C [\![t ]\!]$ an invertible series.
Then for all $F=t^m g(t) \in \OX0$, $g(t)\in \mathbb C [\![t ]\!]$ an invertible series,
we get that
${\rm res}_{0}(\alpha F )=0.$
This implies in particular that $n+m\neq -1$, so $n\in  \mathbb Z \setminus (-\Gamma_X - 1)$.
Let us consider an element
$\alpha=t^n g(t) d t \in \omega_{(X,0)}$, $g(t)\in \mathbb C [\![t ]\!]$ an invertible series,
with  $n\le -c-1$.
Since $-n-1\ge c$ we get $t^{-n-1}\in \OX0$ and then
 a contradiction: $\val_t(\alpha t^{-n-1})=-1$.
The differential  $\alpha =t^{-c} dt$ belongs to $\omega_{(X,0)}$
and $\val_t(\alpha)=-c$.

Let assume now that $(X,0)$ is a monomial curve singularity and let  $n$ be an integer
of $\mathbb Z \setminus (-\Gamma_X - 1)$.
Consider the differential
$\alpha = t^{n} dt$, we only have to prove that $\alpha \in \omega_{(X,0)}$.
Let us consider $F=\sum_{i\ge 0} a_i t^i$ a power series with coefficients $a_i\mathbb C$, $i\ge 0$.
Since $(X,0)$ is monomial, we get that  $F\in \OX0$ if and only if for all $a_i\neq 0$ it holds $i\in \Gamma_X$.
If $\alpha \not\in \omega_{(X,0)}$ then there exists $F\in\OX0$ such that.
${\rm res}_{0}(\alpha F ) =a_{-n-1}\neq 0$.
This implies $-n-1\in\Gamma_X$, which is a  contradiction with the
hypothesis  $n \notin -\Gamma_X  -1$.
\end{proof}

\begin{example}
\label{exp-example}
Let us consider the monomial curve $X$ with parametrization $x_1=t^4, x_2=t^7, x_3=t^9$.
Then $e_0=4$, $c=11$, $\delta= 6$.
Then $\omega_X$ is the $\mathbb C$-vector space   spanned  by
$t^{-11}, t^{-7}, t^{-6}, t^{-4}, t^{-3}, t^{-2}, t^n, n\ge 0$, and the quotient
$\omega_X/ d \OX0$ admits as $\mathbb C$-vector space base the cosets defined by
$t^{-11}, t^{-7}, t^{-6}, t^{-4}, t^{-3},$ $  t^{-2}, 1, t, t^2, t^4, t^5, t^9$.
Notice that  $\mu(X)=12= 2 \delta(X)$.
From \propref{canideal}  we get that $x_1^a \omega_X $ is a canonical ideal
of $X$, $a\ge 101$.
On the other hand  $t^{15} \omega_X= (x_1,x_3)$ is a canonical ideal as well.
\end{example}

The next step
 is to find explicit canonical ideals from the resolution of
$\OX0$ when $X$ is a reduced  curve singularity of  $(\mathbb C^3, 0)$.
By the Hilbert-Burch theorem we know that there exists a minimal free resolution of $\OX0$ as an $\OC30$-module,  \cite{Bur68},
$$
0\longrightarrow \OC30^{v-1}
\stackrel{M}{\longrightarrow} \OC30^v
\longrightarrow \OC30
\longrightarrow \OX0
\longrightarrow 0
$$
where $M$ is a $v\times (v-1)$ matrix with entries belonging to the the maximal ideal $\OC30$ and $I_X$
is minimally generated by the maximal minors of $M$.
The canonical module of $\OX0$ is minimally generated by $v-1$  elements.
Following \cite{Boi99b}, \cite{IK99}, we consider a $(2 v-1) \times (2v -1)$ block-matrix
$$
M_A=\left(
\begin{array}{c|c}
A & M \\ \hline
 -M^\tau &0
\end{array}\right)
$$
where $A=(a_{i,j})_{i,j=1,\cdots, v}$ is a $v\times v$ skew-symmetric matrix.
Notice that $M_A$ is also a skew-symmetric matrix, and by the main result of \cite{BE77}
we have a complex
$$
0\longrightarrow \OC30
\stackrel{\Pf(M_A)^\tau}{\longrightarrow} \OC30^{ 2 v-1}
\stackrel{M_A}{\longrightarrow} \OC30^{2 v-1}
\stackrel{\Pf(M_A)}{\longrightarrow} \OC30
\longrightarrow {\mathcal O}_A =\OC30/I_A
\longrightarrow 0
$$
where $I_A$ is the ideal generated by the Pfaffians $\Pf(M_A)$ of $M_A$.
This complex is exact if and only if $I_Z$ is a height three ideal,   if
this is the case ${\mathcal O}_A$ is a Gorenstein ring, \cite{BE77}.

\begin{proposition}
\label{anti}
$(i)$
Let $A$ be a  $v\times v$ skew-symmetric matrix such that $I_A/ I_X$ is an $\max_X$-primary ideal of  $\OX0$
and  $I_A/ I_X \subset \max_X^{\pn(\OX0)+1} $.
Then $I_A/ I_X$  is a canonical ideal of $(X,0)$.

\noindent
$(ii)$
For all $t\ge \pn(\OX0)$ there exists a $v\times v$ skew-symmetric matrix $A$ with entries of order at least $t$ such that $I_A/ I_X$  is a canonical ideal of $\OX0$.
\end{proposition}
\begin{proof}
$(i)$
This is a consequence of \cite{BE77} and \corref{existence}.

\noindent
$(ii)$
The ideal $I_A$ is generated by the Pfaffians of $M_A$, and these elements take the following form, see \cite{IK99},
\begin{enumerate}
\item $F_i= (-1)^{i + \frac{(v-1)(v-2)}{2} } \det(M_i)$, $i=1,\dots, v$, where $M_i$ is the matrix removing the $i$-th row of $M$,
\item
$F_k=\sum_{1\le i < j \le v} (-1)^{i+j+k + \frac{(v-1)(v-4)}{2} } a_{i,j} \det(M_{i,j,k})$,
$v+1\le k\le 2v-1$,
$M_{i,j,k}$ is the matrix removing the $i$, $j$ rows of $M$ and the $k$ column of $M$.
\end{enumerate}
Notice that $F_1,\cdots, F_v$ is a system of generators of $I_X$.
Hence $I_A/ I_X$ is generated by the cosets of $F_k$ in $\OX0$, $v+1\le k\le 2v-1$.

Let $\pi:\OC30\longrightarrow \OX0$ be the natural projection.
Let $J_X$ be the ideal generated by the $2\times2$ minors of the $v\times 3$ matrix $Jac_X$
whose $i$-th row is the gradient vector
$\nabla F_i=(\partial F_1/\partial x_1, \partial F_1/\partial x_3, \partial F_1/\partial x_3)$,
$i=1,\dots, v$.
The image $\pi(J_X)$ is the Jacobian ideal of $X$.
Since the $(X,0)$ is an isolated curve singularity we have that $\pi(J_X)$ is $\max_X$-primary, so
 $\pi^{-1}\pi(J_X)=I_X+ J_X$ is $(x_1,x_2,x_3)$-primary.

 On the other hand it is easy to prove that
 $$
 I_X+ J_X\subset K=(F_i;  1\le i\le  v; \det(M_{i,j,k}); 1\le i< j\le  v;  v+1\le k\le 2v-1),
 $$
 so $K$ is an $(x_1,x_2,x_3)$-primary ideal.

  Assume that  $\det(M_{i,j,k})$ is zero divisor of $\OX0$ for all $1\le i\le  v; 1\le i< j\le  v;  v+1\le k\le 2v-1)$.
 Let $\frak p_1,\cdots, \frak p_m$ be the set of minimal primes of $\OX0$.
  Then $\frak p_i \neq \max_X$, $i=1,\cdots, m$, and
 $ K/I_X \subset \frak p_1 \cup \cdots \cup \frak p_m.$
 Since $K/I_X$ an $\max$-primary ideal of $\OX0$ there is an integer $w$ such that
 $$
 (x_1,x_2,x_3)^w\subset \frac{K}{I_X} \subset \frak p_1 \cup \cdots \cup \frak p_m.
 $$
 Hence all element of $(x_1,x_2,x_3)^w$ is a zero divisor  of $\OX0$, but this is not
 possible because $\OX0$ is a Cohen-Macaulay ring.
 We have proved that there exist integers
 $1\le i\le  v; 1\le i< j\le  v;  v+1\le k\le 2v-1)$
 such that $\det(M_{i,j,k})$ is  a non-zero divisor in $\OX0$.

 Let $x\in\max\setminus \max^2$ be a non-zerodivisor of $\OX0$.
 Let us consider the skew-symmetric matrix such that
 $$
 \left\{
   \begin{array}{ll}
     a_{i,j}= x^{\pn(\OX0)},  \\
     a_{\alpha,\beta}=0 & \quad  \alpha < \beta, (\alpha,\beta) \neq (i,j).
   \end{array}
 \right.
 $$
Then $F_k= \pm x^{\pn(\OX0)} \det(M_{i,j,k})$ is a non-zero divisor of $\OX0$.
Hence $I_A/I_X$ is an $\max$-primary ideal of $\OX0$ contained in $\max^{\pn(\OX0)+1}$.
By $(i)$ we get that  $I_A/I_X$ is a canonical ideal of $\OX0$.
 \end{proof}

\begin{example}
Let us consider a monomial curve singularity $X$ with parametrization $(t^{n_1},t^{n_2},t^{n_3})$
such that $n_1< n_2< n_3$ and $\gcd(n_1,n_2,n_3)=1$.
Then $e_0(X)=n_1$ and
$\OX0=\mathbb C [\![t^{n_1},t^{n_2},t^{n_3} ]\!] \subset \overline{\OX0}=\mathbb C [\![t ]\!]$.
The  ideal $I_X$ is generated by
$F_1= x_1^{r_{3,1}} x_2^{r_{3,2}}- x_3^{c_3}$, $F_2= x_2^{r_{1,2}} x_3^{r_{1,3}}- x_1^{c_1}$ and $F_3= x_1^{r_{2,1}} x_3^{r_{2,3}}- x_2^{c_2}$,
where $r_{i,j}\ge 0$ and $c_i>0$ is the least integer such that
$c_i n_i = \sum_{i\neq j} r_{i,j} n_j$, $i=1,2,3$, \cite{Her70}.
We assume that $X$ is not a complete intersection, so $r_{i,j}>0$.
Then we have that $c_1= r_{2,1} + r_{3,1}$,
$c_2= r_{1,2} + r_{3,2}$,
$c_3= r_{1,3} + r_{2,3}$,
and $F_1, F_2, F_3$ are the maximal minors of a matrix
$$
M=\left(
    \begin{array}{cc}
      x_1^{r_{2,1}} & x_2^{r_{1,2}}  \\ \\
      x_2^{r_{3,2}} & x_3^{r_{2,3}}  \\ \\
      x_3^{r_{1,3}} & x_1^{r_{3,1}}
    \end{array}
  \right)
$$
\cite{RV80}.
Then the matrix $M_A$ takes the from
$$
M_A=
C=\left(
\begin{array}{c|c}
A &M \\ \hline
-M^{\tau} &0
\end{array}\right)
=
\left(
\begin{array}{c|c}
\begin{array}{ccc}
     0& a_{1,2} & -a_{1,3} \\ \\
    -a_{1,2} & 0& a_{2,3} \\ \\
    a_{1,3} & -a_{3,2}&0\\ \\
\end{array}
 &
\begin{array}{cc}
       x_1^{r_{2,1}} & x_2^{r_{1,2}}  \\  \\
      x_2^{r_{3,2}} & x_3^{r_{2,3}}  \\ \\
      x_3^{r_{1,3}} & x_1^{r_{3,1}} \\ \\
\end{array}
\\ \hline \\
\begin{array}{ccc}
      -x_1^{r_{2,1}} & -x_2^{r_{3,2}} & - x_3^{r_{1,3}}\\ \\
      -x_2^{r_{1,2}} & -x_3^{r_{2,3}} & -x_1^{r_{3,1}}\\ \\
\end{array}
&
\begin{array}{cc}
      0 & 0  \\ \\
      0 & 0 \\  \\
\end{array}
\end{array}\right)
$$
with
$a_{i,j}\in \max^{\pn(\OX0)}$.
Then $I_A/I_X$ is a canonical ideal generated by the cosets in $\OX0$ of
$ a_{2,3} m_{1,i}+ a_{1,3} m_{2,i} + a_{1,2} m_{3,i}$, $i=1,2$, if one of this two elements
is a non-zerodivisor of $\OX0$, see \propref{anti} $(ii)$.
For instance, we can take $a_{2,3}=x_1^{\pn(X)}$ and $a_{i,j}=0$ for $(i,j)\neq (2,3)$, so $(x_1^{r_{2,1}+\pn(X)}, x_1^{\pn(X)} x_2^{r_{1,2}})$ is a canonical ideal of $X$.
\end{example}


\section{Canonical ideals and the classification of curve singularities}

In this section  we translate the classification of curve singularities to the classification
of local Artin Gorenstein rings by means of the quotients by canonical ideals.

First we have to define what generic means in our context.
We denote by $S_t$ the $\mathbb C$-vector space of forms of degree $t$ of $\series{n}$.

\begin{proposition}
\label{generic}
\noindent
$(i)$  For all $t\ge 2 \mu(X)+1 $ there exists a non-empty Zariski open set $U_t(X)$ of $S_t$ such that:
\begin{enumerate}
\item[(i.1)]
for all $z \in U_t(X)$,
$\overline{z} \in \OX0$ is a degree $t$ superficial element and non-zerodivisor,

\item[(i.2)]  for all $z_1, z_2 \in U_t(X)$
$$
\HF^1_{\OX0/\overline{z_1} \omega_{(X,0)}}=
\HF^1_{\OX0/\overline{z_2} \omega_{(X,0)}}
$$

\item[(i.3)]
 for every degree $t$ superficial element   $y\in \OX0$ and for every
 $z\in U_t(X)$ it holds that for all   $n\ge 0$,
$$
\HF^1_{\OX0/\overline{z_1} \omega_{(X,0)}}(n)\le
\HF^1_{\OX0/y \omega_{(X,0)}}(n).
$$
\end{enumerate}

\noindent
$(ii)$
For  all
$z \in U_t(X)$ and  $t\ge 4 \mu(X)+1$, the socle degree and the length  of $\OX0/\overline{z} \omega_{(X,0)}$
are constant and  analytic invariants of $(X,0)$  and satisfy
$$
s(\OX0/ \overline{z}\omega_{(X,0)})\le e_0(X)(t-2 \mu(X)-1) + 2 \delta(X) +e_1(X)+ 2(1-r),
$$
$$
\length(\OX0/ \overline{z}\omega_{(X,0)})=e_0(X) t- 2 \delta(X).
$$
\end{proposition}
\begin{proof}
$(i)$
From \cite[Proposition 3.2]{Sal78} there exists a non-empty Zariski open set $W_0\subset S_t$ such that  $ \overline{z} \in \OX0$ is a degree $t$ superficial element and non-zerodivisor.
From  \propref{canideal} $(i)$ we know that for all $z \in W_0$ we have
$\length(\OX0 / \overline{z}\omega_{(X,0)} )= \ell$  with $\ell=t  e_0(X) -  2 \delta(X)$,
so $\HF^1_{\OX0 /\overline{z}\omega_{(X,0)} }(n)=\ell$  for all $n\ge \ell$.
Hence we only have to consider $n\le \ell$.
Let us consider the upper-semicontinous function
$$
\begin{array}{cccc}
  \sigma_n :& S_t & \longrightarrow & {\mathbb N}\\
   & z  & \mapsto & \HF^1_{\OX0/\overline{z}\omega_{(X,0)}}(n).
\end{array}
$$
For each $n=1,\cdots, \ell$,  let $W_n$ be a non-empty Zariski open set $W_n\subset W_0$ such that
$\sigma_n(z)=\min \{\sigma_n\}$ for all $z \in W_n$.
We set $U_t(X)=W_0\cap \cdots \cap W_{\ell}$.
From the definition of $U_t(X)$ it is easy to get $(2)$ and $(3)$.

\noindent
$(ii)$
From $(i.2)$ we deduce that
the socle degree and the length  of $\OX0/\overline{z}\omega_{(X,0)}$ are constants for
all $t\ge 2 \mu (X) +1  $ and $ z \in U_t(X)$.
The upper bounds come from \propref{canideal}.
Let $\varphi: (X_1,0) \rightarrow (X_2,0)$ be an analyic isomorphism between  two reduced curve singularities.
Let $\varphi_t : \max_{X_1}^t/\max_{X_1}^{t+1}\rightarrow \max_{X_2}^t/\max_{X_2}^{t+1}$ be the $\mathbb C$-vector space  isomorphism induced by $\varphi$.
Since $\varphi_t(U_t(X_1))\cap U_t(X_2) \neq\emptyset $ we get the last part of the claim.
\end{proof}

 We write $\ell(X)= e_0(X)(4 \mu(X)+1)- 2 \delta(X)$.
 Notice that this is the length
 of the quotients $\OX0/\overline{z}\omega_{(X,0)}$ for $\overline{z} \in U_{4 \mu(X)+1}(X)$.

\begin{definition}
We denote by $\sigma(X)$ the least socle degree of the quotients $\OX0 /I$ where
$I$ range the set of canonical ideals $I\subset \OX0$ with
$I\subset \max^{2 \mu(X)+1}$ and  $\length(\OX0 /I)=\ell(X)$.
A canonical ideal $I$ is called {\bf deep} if
$I\subset \max^{2 \mu(X)+1}$,
$\length(\OX0 /I)=\ell(X)$ and
$s(\OX0 /I)=\sigma(X)$.
Notice that from \propref{generic} deep canonical ideals exist and that $\sigma(X)$ is an analytic invariant.
\end{definition}

\begin{remark}
If $(X,0)$ is non-singular then we can take $t=1$ in the last identity of \propref{generic};
in fact we have
$s(\OX0/ x \OX0)=s(\res)= 0.$
If $(X,0)$ is Gorenstein then  we can take as canonical ideal the whole ring $I=\OX0$.
Then we have $x\in \OX0$ degree one superficial element, that
$$
s\left(\frac{\OX0}{x^t \OX0}\right)=
s\left(\frac{\OX0}{x^{e_0(X)-1}\OX0}\right) + t-e_0(X)+1,
$$
and
$\length(\OX0/x^t \OX0)=e_0(X) t,$
for all $t\ge e_0(X)-1$.
\end{remark}

Given a non-negative integer $t$ we denote by $\hilb_{({\mathbb C}^n,0)}^t$ the Hilbert scheme of length $t$
subschemes $Z$ of $({\mathbb C}^n,0)$.
We denote by  $[Z]$ is the closed point of  $\hilb_{({\mathbb C}^n,0)}^t$ defined by $Z$.
From the universal property of $\hilb_{({\mathbb C}^n,0)}^t$ we deduce that any
analytic isomorphism $\phi: ({\mathbb C}^n, 0) \longrightarrow  ({\mathbb C}^n, 0)$ induces
a $\mathbb C$-scheme isomorphism
$\tilde \phi: \hilb_{({\mathbb C}^n,0)}^t\longrightarrow \hilb_{({\mathbb C}^n,0)}^t$
such that $\tilde \phi([Z])=[\phi(Z)]$.
Given a canonical ideal $I$ of a reduced curve singularity $(X,0)$ we denote by
$(X,0)_I$ the zero-dimensional scheme $\spec(\OX0 /  I)$.
We know that $(X,0)_I$ is an Artin  Gorenstein scheme, \cite[Proposition 3.3.18]{BH97}.
It is well-known that two  canonical ideals are isomorphic, \cite[Theorem 3.3.4]{BH97}.
In the one-dimensional case one can prove more:  if $I_1$, $I_2$ are canonical ideal
there exist non-zero divisors $y_1, y_2 \in\OX0$ such that
$y_1 I_1 =y_2 I_2$, \cite[Theorem 15.8]{Mat77}.
Notice that from the proof of this result  we get that for all  $z\in U_t(X)$ there exists  integer $\alpha$
and a regular  element $y$
such that $z^{\alpha} I= y z \omega_{(X,0)}$.
The ideals $I_1$, $I_2$ are isomorphic as $\OX0$-modules,  but, in general, there are no  analytic isomorphic.
Since $K_t=x^{ t}\omega_{(X,0)}$ is a canonical ideal for all $t\ge 1$, $x$ a degree-one superficial element,  the Hilbert function
of $\OX0/K_t$ varies with $t$.

\begin{proposition}
\label{canlocus}
There exists a non-empty subscheme $\Can(X,0)$ of the Hil\-bert sche\-me  $\hilb_{({\mathbb C}^n,0)}^{\ell(X)}$ whose closed points correspond to  zero-dimensional Gorenstein schemes $(X,0)_I \subset ({\mathbb C}^n,0)$ of length $\ell(X)$ and
such that $I\subset \max_X^{2 \mu(X)+1}$ is a  deep canonical ideal.
\end{proposition}
\begin{proof}
Notice that by \propref{canideal} there exist canonical ideals  $I\subset \max_X^{2 \mu(X)+1}$  of colength  $\ell(X)$.
Hence the set of closed points of $\hilb_{({\mathbb C}^n,0)}^{\ell(X)}$ corresponding to
such a canonical ideal is non-empty.
On the other hand, by standard arguments on the semi-continuity of the dimension of $\res$-vector spaces, there exists a sub-scheme of $\hilb_{({\mathbb C}^n,0)}^{\ell(X)}$ such that its closed points correspond to the quotients  $\OX0/I$ with $I\subset \max_X^{2 \mu(X)+1}$.
Since $\OX0$ is a one-dimensional Cohen-Macaulay ring, $I$ is a faithful maximal Cohen-Macaulay $\OX0$-module.
Hence from \cite[Proposition 3.3.13]{BH97} we get that $I$ is a canonical ideal iff  $I$ is a
type one Cohen-Macaulay   $\OX0$-module.
Since   the Cohen-Macaulay type is a  positive upper semi-continuous function,
we get the claim.
\end{proof}

\begin{theorem}
\label{chargor}
Given reduced singularities $(X_i,0)$, $i=1, 2$,  the following conditions are equivalent:
\begin{enumerate}
\item[(i)]
there exists an  analytic isomorphism  $\phi: ({\mathbb C}^n, 0) \longrightarrow  ({\mathbb C}^n, 0)$ such that $\phi(X_1,0)=(X_2,0)$,

\item[(ii)]
there exists an  analytic isomorphism  $\phi: ({\mathbb C}^n, 0) \longrightarrow  ({\mathbb C}^n, 0)$
inducing a $\mathbb C$-scheme isomorphism
$\tilde \phi: \Can(X_1,0) \longrightarrow  \Can(X_2,0)$,

\item[(iii)]
there exists an  analytic isomorphism  $\phi: ({\mathbb C}^n, 0) \longrightarrow  ({\mathbb C}^n, 0)$ and  $[(X_1,0)_I]\in \Can(X_1,0)$ such that
$[\phi((X_1,0)_I)]\in \Can(X_2,0)$.
\end{enumerate}
\end{theorem}
\begin{proof}
Given an integer $s\ge 1$ we  denote by $(X,0)_s$ the Artin scheme defined by $\OX0/\max_X^s$.
The implications $(i)\Rightarrow (ii)\Rightarrow(iii)$ are trivial.
Assume that there exist $[(X_1,0)_I]\in \Can(X_1,0)$ such that
$[\phi((X_1,0)_I)]\in \Can(X_2,0)$.
If $\mu(X_1)\le \mu(X_2)$ then, since $I\subset \max_X^{2 \mu(X)+1}$ we get that
$(\phi(X_1,0))_{2\mu(X_1)+1}=(X_2,0)_{2\mu(X_1)+1}$.
From the main result of \cite[Theorem 6]{Eli86c} we get that
there exists an  analytic isomorphism  $\varphi: \Cn0 \longrightarrow  \Cn0$
such that
$\varphi \phi (X_1,0)= (X_2,0).$
If $\mu(X_2)< \mu(X_1)$ then we consider
$(\phi^{-1}(X_2,0))_{2\mu(X_2)+1}=(X_1,0)_{2\mu(X_2)+1}$.
Again, from the main result of \cite[Theorem 6]{Eli86c} we get $(i)$.
\end{proof}

An instance of Matlis duality is Macaulay's inverse system:
given a Gorenstein Artin algebra $B$ of socle degree $s$, the Macaulay inverse system of $B$
is  a polynomial of degree $s$ that encodes several algebraic properties of $B$,  see \cite{Iar94}.
In the next result we find canonical ideals $I\subset \OX0$, with suitable socle degree, as the first step to consider the inverse system of a curve singularity.

\begin{proposition}
For all $z \in U_{4 \mu(X) +1}(X)$,  $x\in U_1(X)$  and $a\ge 0$ it holds
$$
s\left(\frac{\OX0}{x^{a+1} \overline{z} \omega_{(X,0)}}\right)=
1+ s\left(\frac{\OX0}{x^{a} \overline{z}\omega_{(X,0)}}\right).
$$
In particular for suitable integer $a$ we have
$$
s\left(\frac{\OX0}{x^{a} \overline{z} \omega_{(X,0)}}\right)= \delta(X) (4 e_0(X)+3).
$$
\end{proposition}
\begin{proof}
The proof of the first identity is standard.
The second one follows from \propref{canideal} $(ii)$ and \propref{basic}.
\end{proof}

Notice that from the last result we can attach to a curve singularity $(X,0)$ an Artin Gorenstein
local ring $B_X=\OX0/x^{a} z \omega_{(X,0)}$ with socle degree $\delta(X) (4 e_0(X)+3)$.
Consequently   we can attach to a curve singularity the Macaulay's inverse system $\xi_X$ of $B_X$
that is a degree $\delta(X) (4 e_0(X)+3)$ polynomial.
Hence the algebraic and geometric structure of $X$ is encoded in the polynomial $\xi_X$.
The development of this fact will be considered elsewhere.

\begin{example}
Let us consider the monomial ring $R=\res[\![t^4,t^7,t^9]\!]$ of \exref{exp-example}.
In  this case the polynomial $\xi_X$ has degree $144$.
 The explicit computation of this polynomial seems to be very hard, but  we can consider  a more friendly canonical ideal.
Notice that $R=\res[\![x_1,x_2,x_3]\!]/I$ with $I=(x_1^4-x_2x_3, x_2^3-x_1^3x_3,x_3^2- x_1 x_2^2)$, and  $J=t^{15} \omega_X= (x_1,x_3)$ is a canonical ideal.
Since $  x_1^{24}(x_1,x_3)\subset \max^{2\mu(X)+1}$, the analytic type of $X$ is determined by  the analytic type of the Artin Gorenstein algebra $B=R/ x_1^{24}(x_1,x_2)$, \cite{Eli86c}.
The ring $B$ has  Hilbert function $\{1,3,4^{(23)},2,1\}$, so that $B$ is of multiplicity $99$ and socle degree $26$,  the inverse system is the degree $26$ polynomial:

\medskip
\noindent
$
31087081215590400x_1x_2x_3^{11}+42744736671436800x_2^{8}x_3^{6}+
284964911142912000x_1^{2}x_2^{3}x_3^{9}+$

\noindent
$
14248245557145600x_1x_2^{10}x_3^{4}+
341957893371494400x_1^{3}x_2^{5}x_3^{7}+2849649111429120x_1^{5}x_3^{10}+
$

\noindent
$
647647525324800x_1^{2}x_2^{12}x_3^{2}+85489473342873600x_1^{4}x_2^{7}x_3^{5}+
21372368335718400x_1^{6}x_2^{2}x_3^{8}+
$

\noindent
$
2372335257600x_1^{3}x_2^{14}+
4749415185715200x_1^{5}x_2^{9}x_3^{3}+14248245557145600x_1^{7}x_2^{4}x_3^{6}+
43176501688320x_1^{6}x_2^{11}x_3+1781030694643200x_1^{8}x_2^{6}x_3^{4}+
67848788367360x_1^{10}x_2x_3^{7}+42405492729600x_1^{9}x_2^{8}x_3^{2}+
43176501688320x_1^{11}x_2^{3}x_3^{5}+94234428288x_1^{10}x_2^{10}+
3598041807360x_1^{12}x_2^{5}x_3^{3}+19769460480x_1^{14}x_3^{6}+
39538920960x_1^{13}x_2^{7}x_3+
$

\noindent
$
19769460480x_1^{15}x_2^{2}x_3^{4}+
1235591280x_1^{16}x_2^{4}x_3^{2}+4845456x_1^{17}x_2^{6}+1700160x_1^{19}x_2x_3^{3}+
85008x_1^{20}x_2^{3}x_3+24x_1^{23}x_3^{2}+x_1^{24}x_2^{2}
$.
\end{example}

\providecommand{\bysame}{\leavevmode\hbox to3em{\hrulefill}\thinspace}
\providecommand{\MR}{\relax\ifhmode\unskip\space\fi MR }
\providecommand{\MRhref}[2]{%
  \href{http://www.ams.org/mathscinet-getitem?mr=#1}{#2}
}
\providecommand{\href}[2]{#2}

\noindent
Juan Elias\\
Departament d'\`Algebra i Geometria\\
Universitat de Barcelona\\
Gran Via 585, 08007 Barcelona, Spain\\
e-mail: {\tt elias@ub.edu}

\end{document}